\documentclass[preprint,11pt]{elsarticle}
\usepackage{amsfonts}
\usepackage{amsmath}
\usepackage{mathrsfs}
\usepackage[applemac]{inputenc}
\usepackage{amssymb}
\usepackage{times, ulem, amsfonts}
\usepackage{mathpazo, amsmath, amssymb, amsthm,color}

\newcommand{\p}{{\mathbb P}}

\def\virgp{\raise 2pt\hbox{,}}
\def\({\left(}
\def\){\right)}

\def\<{\langle}
\def\>{\rangle}

\theoremstyle{plain}

\newtheorem{theorem}{Theorem}[section]
\newtheorem{lemma}[theorem]{Lemma}
\newtheorem{definition}[theorem]{Definition}
\newtheorem{corollary}[theorem]{Corollary}

\newtheorem{remark}[theorem]{Remark}
\numberwithin{equation}{section}

\newcommand{\R}{{\mathbb R}}
\def\div{ \hbox{\rm div}\,  }
\newcommand\Z{{\mathbb{Z}}}
\def\La{\Lambda}
\def\ga{\Gamma}
\def\aa{\phi}
\def\ddj{\dot \Delta_j}
\def\f{\frac}
\begin{document}
\bibliographystyle{plainmma}
\baselineskip=18pt
\footnote{
Email Addresses: zhaixp@szu.edu.cn (X. Zhai).}
\begin{center}
{\Large\bf  Global  solutions to the  $n$-dimensional incompressible Oldroyd-B model without damping mechanism}\\[1ex]

Xiaoping Zhai \\[1ex]
 School  of Mathematics and Statistics, Shenzhen University,
 Shenzhen 518060, China
\end{center}

\centerline{\Large\bf Abstract}
The present work is dedicated to the global solutions to the incompressible Oldroyd-B model without damping on the stress tensor in $\mathbb{R}^n(n=2,3)$. 
This result
allows to construct global solutions for a class of highly oscillating
initial velocity.
The proof uses the special structure of the system.
Moreover, our theorem extends the previous result by  Zhu \cite{zhuyi} and  covers the recent  result  by Chen and Hao  \cite{chenqionglei}.

\date{\today}



\setcounter{section}{1}
\setcounter{theorem}{0}
\section*{\Large\bf 1. Introduction and the main result}
In this paper, we mainly consider the  incompressible Oldroyd-B model without damping mechanism which has the following form:
\begin{eqnarray}\label{m}
\left\{\begin{aligned}
&\partial_t\tau + u\cdot \nabla \tau   + F(\tau, \nabla u)  - K_2 D (u)=0,\\
&\partial_t u+ u\cdot\nabla u-\mu\Delta u+\nabla\Pi-K_1  \div \tau=0,\\
&\div u =0,\\
&(u,\tau)|_{t=0}=(u_0,\tau_0),
\end{aligned}\right.
\end{eqnarray}
where
 $u=(u_{1},u_{2},\cdot\cdot \cdot, u_{n})$ denotes the velocity, $\Pi$ is the scalar pressure of fluid. $\tau=\tau_{i,j}$ is the non-Newtonian part of stress tensor which can be seen as a symmetric matrix here. $D(u)$ is the symmetric part of $\nabla u$,
\begin{equation*}
D(u) = \frac{1}{2} \big( \nabla u + (\nabla u)^{T} \big),
\end{equation*}
and $F $ is a given bilinear form which can be chosen as
\begin{equation*}
F(\tau, \nabla u)= \tau \Omega(u) - \Omega(u) \tau + b(D(u) \tau + \tau D(u)),
\end{equation*}
where $b$ is a parameter in $[-1,1]$, $\Omega(u)$ is the skew-symmetric part of $\nabla u$, namely
\begin{equation*}
\Omega(u) = \frac{1}{2} \big( \nabla u - (\nabla u)^{T} \big).
\end{equation*}
The coefficients $\mu,  K_1, K_2$ are assumed to be non-negative constants.

In fact, the above system \eqref{m} is only the subsystem of the following full incompressible Oldroyd-B model:
\begin{eqnarray}\label{mmm}
\left\{\begin{aligned}
&u_t + u\cdot \nabla u -  \mu \Delta u + \nabla \Pi =  K_1 \div \tau,  \\
&\tau_t + u\cdot \nabla \tau  -  \eta \Delta\tau + \beta \tau + F(\tau, \nabla u) =  K_2 D (u),\\
&\div u = 0,\\
&(u,\tau)|_{t=0}=(u_0,\tau_0),
\end{aligned}\right.
\end{eqnarray}
in which $\eta $ and $ \beta$  are two non-negative constants.

The Oldroyd-B model describes the motion of some viscoelastic flows, for example, the system coupling fluids and polymers. It presents a typical constitutive law which does not obey the Newtonian law (a linear relationship between stress and the gradient of velocity
in fluids). Such non-Newtonian property may arise from the memorability of some fluids. Formulations about viscoelastic flows of Oldroyd-B type are first introduced by Oldroyd \cite{Oldroyd} and are extensively discussed in \cite{bird}.

About the derivation of the system \eqref{mmm}, the interested readers can refer to \cite{lin2012}, here we omit it.
As one of the most popular constitutive laws, Oldroyd-B model of viscoelastic fluids has attracted many attentions and lots of excellent works have been done (see \cite{chemin}, \cite{chenqionglei}, \cite{constanin},
\cite{EL}, \cite{ER}, \cite{FZ}, \cite{GS}, \cite{GS2}, \cite{LMZ}, \cite{LZ}, \cite{zhuyi}, \cite{zuiruizhao}) and references therein.
Guillop\'{e} and Saut \cite{GS}, \cite{GS2} got the local well-posedness   with large initial data and global well-posedness provided that the coupling parameter and  initial data are small enough. Lions and Masmoudi \cite{LM} got the global existence of weak solutions in the corotational case ($b = 0$). However, the case $b \neq 0$ is still not clear by now.
 In the framework of the  near critical  Besov spaces, Chemin and Masmoudi \cite{chemin} first studied the local solutions and global small solutions of system \eqref{mmm} when $ \mu>0$, $K_1>0$, $K_2>0$, $\eta=0$, $\beta>0$. Zi, Fang and Zhang  \cite{zuiruizhao}  improved the result obtained by Chemin and Masmoudi  in   \cite{chemin} to the non-small coupling parameter case.
Recently, Elgindi and Rousset  \cite{ER} proved the global small solutions to  \eqref{mmm} with $\mu=0,K_1,K_2,\eta,\beta>0$  in Sobolev  space $H^s(\R^2), s>2$.
Moreover, if neglect the effect of the quadratic form $F(\tau, \nabla u)$ and let $\mu=0,K_1\ge0,K_2\in\R,\eta>0,\beta\ge0$, they also got the global solutions without any smallness imposed on the initial data in $\R^2$.
Later on, Elgindi and Liu  \cite{EL} consider the global well-posedness of system \eqref{mmm} in $\R^3.$
When  $\mu=0$, $K_1\ge0$, $K_2\in\R$, $\eta>0$, $-\beta>0$,  they obtained the global small solutions in Sobolev spaces
 $H^s(\R^3)$, $s>{5}/{2}$.
 Let us emphasis that the results obtained in \cite{EL}, \cite{ER}, \cite{FZ}, \cite{LM}, \cite{zuiruizhao} always require $\beta > 0$ in \eqref{mmm} (namely the system with damping) at least for non-trivial initial data.
Thus,  it's an interesting problem to study the global well-posedness when $\mu>0$, $\eta=0$, $\beta=0$, $K_1>0$, $K_2>0$ in \eqref{mmm} in $\R^n(n=2,3)$.
Most recently,
Zhu \cite{zhuyi} obtained the global small  solutions to the three-dimensional incompressible Oldroyd-B model without damping on the stress tensor (i.e. $\beta=0$), more precisely, the author in \cite{zhuyi} proved the following theorem.
\begin{theorem}\label{thm}(see \cite{zhuyi})
Let $n=3$, $\mu$,  $K_1$, $K_2 >0$. Suppose that $\div u_0 = 0, (\tau_0)_{ij} =( \tau_0)_{ji}$ and initial data 
$\varepsilon$ such that system \eqref{m} admits a unique global classical solution provided that
$$ \big\|\La^{-1}u_0\big\|_{H^3} + \big\||\La^{-1}\tau_0\big\|_{H^3} \leq \varepsilon, $$
where $\La = (-\Delta)^\frac{1}{2}$.
\end{theorem}

However, the method used by Zhu in \cite{zhuyi} is not valid for $n=2.$  Recently, Chen and Hao  \cite{chenqionglei} generalized the result by
Zhu in \cite{zhuyi} to the critical Besov spaces.
The aim of the present paper is to establish the global solutions of  \eqref{m} with a class of highly oscillating
initial velocity.

In all that follows, let $\mu= K_1=K_2=1$ in \eqref{m}, $\La \stackrel{\mathrm{def}}{=} (-\Delta)^\frac{1}{2}$, we shall denote the projector by $\p=\mathcal{I}-\mathcal{Q}\stackrel{\mathrm{def}}{=}\mathcal{I}-\nabla\Delta^{-1}\div$.

Now, we can  state the  main theorem of the present paper:
\begin{theorem}\label{zhuyaodingli}
Let   $ n=2,3$ and
\begin{equation*}
2\leq p \leq \min(4,{2n}/({n-2}))\quad\hbox{and, additionally, }\  p\not=4\ \hbox{ if }\ n=2.
\end{equation*}  For any $(u_0^\ell,\tau_0^\ell)\in \dot{B}_{2,1}^{\frac n2-1}(\R^n)$, $u_0^h\in \dot{B}_{p,1}^{\frac np-1}(\R^n)$,  $\tau_0^h\in \dot{B}_{p,1}^{\frac np}(\R^n)$ with $\div u_0 = 0$.
 If  there exists a positive constant $c_0$ such that,
\begin{align}\label{smallness2}
\|(u_0^\ell,\tau_0^\ell)\|_{\dot{B}_{2,1}^{\frac {n}{2}-1}}+ \|u_0^h\|_{\dot B^{\frac  np-1}_{p,1}}+\|\tau_0^h\|_{\dot B^{\frac  np}_{p,1}}\leq c_0,
\end{align}
then
the system \eqref{m} has a unique global solution $(u,\tau)$ so that for any $T>0$
\begin{align*}
&u^\ell\in C_b([0,T );{\dot{B}}_{2,1}^{\frac {n}{2}-1}(\R^n))\cap L^{1}
([0,T];{\dot{B}}_{2,1}^{\frac n2+1}(\R^n)),\\
&
\tau^\ell\in C_b([0,T );\dot{B}_{2,1}^{\frac n2-1}(\R^n)), \quad(\Lambda^{-1}\p\div\tau)^\ell\in L^{1}
([0,T];{\dot{B}}_{2,1}^{\frac n2+1}(\R^n)),\\
&u^h\in C_b([0,T );{\dot{B}}_{p,1}^{\frac {n}{p}-1}(\R^n))\cap L^{1}
([0,T];{\dot{B}}_{p,1}^{\frac np+1}(\R^n)),\\
&
\tau^h\in C_b([0,T );\dot{B}_{p,1}^{\frac np}(\R^n)), \quad (\Lambda^{-1}\p\div\tau)^h\in L^{1}
([0,T];{\dot{B}}_{p,1}^{\frac np}(\R^n)).
\end{align*}
\end{theorem}

\begin{remark}\label{23}
By a similar argument as Zhu in \cite{zhuyi}, treating the nonlinear term to linear term, we can also get the global small solutions for
 the
incompressible viscoelastic system with Hookean elasticity.
\end{remark}

\begin{remark}\label{2333}
Most recently, Chen and Hao in \cite{chenqionglei} get the global well-posedness of \eqref{m} in $\R^n$ with $ n\ge2.$ Compared with  Chen and Hao in \cite{chenqionglei}, the global solutions we constructed here allow the highly oscillating
initial velocity.  A typical example is
\begin{align*}
u_0(x)=\sin\bigl(\frac {x_1} {\varepsilon}\bigr)\phi(x), \quad \phi(x)\in \mathcal{S}(\R^n),\quad  p>n
\end{align*}
which satisfies for any $\varepsilon>0$
\begin{align*}
\|u_0^\ell\|_{\dot{B}_{2,1}^{\frac {n}{2}-1}}+ \|u_0^h\|_{\dot B^{\frac  np-1}_{p,1}}\le C\varepsilon^{1-\frac np},
\end{align*}
here C is a constant independent of $\varepsilon$ (see [\cite{miaochangxing}, Proposition 2.9]).
\end{remark}

\begin{remark}\label{123}
 Compared with the result obtained by Chemin and Masmoudi in \cite{chemin}, we also obtain the global small solutions, yet there is no  damping mechanism.
\end{remark}
\begin{remark}
Our methods can be used to other related models. Similar results for the compressible Oldroyd-B model will be given in a forthcoming paper.
\end{remark}


\noindent $\mathbf{ Scheme\  of \ the\  proof\  and\  organization \ of\  the\   paper.}$ The main difficulty to the proof of Theorem \ref{zhuyaodingli} lies in the fact that there is no  dissipation in stress tensor. Thus, we can not get directly any integration for stress tensor $\tau$ about time in the basic energy argument. Indeed, we also can not get any  integration about time of $u$. One can see more detail  in the derivation of \eqref{caihong7} in the third section.

To exploit the dissipation of $u$ and to find the partial dissipation hidden for $\tau$, let us  first study the linear system of \eqref{m} (without loss of generality, set $\mu = K_1 = K_2 = 1$).

Applying project operator $\p$ on both hand side of the first two equation in \eqref{m} gives
\begin{eqnarray}\label{QG1}
\left\{\begin{aligned}
&\partial_t  u -\Delta  u-\p \div \tau=G_1,\\
&\partial_t \p \div \tau -\Delta u=G_2.
\end{aligned}\right.
\end{eqnarray}

At the linear level,
to weaken the effect of $\Delta u$ appeared in the stress tensor equation, we introduce
 $\aa\stackrel{\mathrm{def}}{=}\Lambda^{-1}\p\div\tau$ with $\Lambda \stackrel{\mathrm{def}}{=}(-\Delta )^{\frac12}$,  a simple computation from \eqref{QG1} gives
\begin{equation}\label{QGQ1}
\left\{\begin{array}{l}
\partial_t\aa+\La u=\Lambda^{-1}G_2,\\
\partial_t u-\Delta u-\La\aa=G_1.
\end{array}\right.\end{equation}
The above system is
similar to the linear system of the compressible Navier-Stokes equations \cite{danchin2000}.
In the following, we recall the analysis of the linearized system \eqref{QGQ1}.
Taking the Fourier transform with respect to $x,$
 System  \eqref{QGQ1} translates  into
\begin{equation}\label{GQ11}
\frac d{dt}\left(\begin{array}{c}\hat{\aa}\\\hat{u}\end{array}\right)=
A(\xi)\left(\begin{array}{c}\hat{\aa}\\\hat{u}\end{array}\right)+\left(\begin{array}{c}\widehat{{\Lambda^{-1}G_2}}\\\widehat{{G_1}}\end{array}\right)  \quad\hbox{with}\quad
A(\xi)\stackrel{\mathrm{def}}{=} \left(\begin{array}{cc}0&-|\xi|\\ |\xi|&-|\xi|^2\end{array}\right)\cdotp
\end{equation}
\begin{itemize}
\item In the low frequency regime $|\xi|<2,$ $A(\xi)$ has  two complex conjugated eigenvalues:
$$
\lambda_\pm(\xi)\stackrel{\mathrm{def}}{=} -\frac{\xi^2}{2}\left(1\pm i \sqrt{\frac{4-\xi^2}{\xi^2}}\right)
$$
which have real part $-\frac{\xi^2}{2},$ exactly as for the heat equation with diffusion $\frac12.$
\item  In the  high frequency regime $|\xi|>2,$ there are two distinct real eigenvalues:
$$
\lambda_\pm(\xi)\stackrel{\mathrm{def}}{=} -\frac{\xi^2}{2}\left(1\pm \sqrt{\frac{\xi^2-4}{\xi^2}}\right).
$$
 As $1+\sqrt{\frac{\xi^2-4}{\xi^2}}\sim  2$  and $1-\sqrt{\frac{\xi^2-4}{\xi^2}}\sim  \frac{2}{\xi^2}$ for $\xi\to+\infty,$ we can deduce that
 $\lambda_+(\xi)\sim -\xi^2$ and $\lambda_-(\xi)\sim -1\cdotp$
In other words, a parabolic and a damped mode coexist.
\end{itemize}
Optimal a priori estimates may be easily derived by computing
the explicit solution of  \eqref{GQ11} explicitly in the Fourier space.

In the second section, we shall collect some basic facts on Littlewood-Paley analysis and various product laws in
Besov spaces. In  Section 3, we will use three subsections to prove the main Theorem \ref{zhuyaodingli}, we apply the Littlewood-Paley theory to get the basic energy estimates for $(u,\tau)$,  and then  by introducing a new quantity, we get the low frequency and high frequency of the solutions of $(a, \p\div \tau)$ in the first subsection and  the second subsection, respectively.
 Finally in the last subsection, we present the proof to the global well-posedness of Theorem \ref{zhuyaodingli} by  standard continuous argument.

\noindent$\mathbf{Notations:}$ Let $A$, $B$ be two operators, we denote $[A, B] = AB - BA$, the commutator
between $A$ and $B$. For $a\lesssim b$, we mean that there is a uniform constant $C$, which may
be different on different lines, such that $a \le C b$. We shall denote by$\langle a,b\rangle$ the $L^2(\R^n)$ inner product of $a$ and $b$.
\setcounter{section}{2}
\setcounter{theorem}{0}
\setcounter{equation}{0}
\section*{\large\bf 2. Preliminaries}
The {Littlewood-Paley decomposition}  plays a central role in our analysis.
To define it,   fix some  smooth radial non increasing function $\chi$
supported in the ball $B(0,\frac 43)$ of $\R^n,$ and with value $1$ on, say,   $B(0,\frac34),$ then set
$\varphi(\xi)=\chi(\frac{\xi}{2})-\chi(\xi).$ We have
$$
\qquad\sum_{j\in\Z}\varphi(2^{-j}\cdot)=1\ \hbox{ in }\ \R^n\setminus\{0\}
\quad\hbox{and}\quad \mathrm{Supp}\,\varphi\subset \Big\{\xi\in\R^n : \frac34\leq|\xi|\leq\frac83\Big\}\cdotp
$$
The homogeneous dyadic blocks $\dot{\Delta}_j$ are defined on tempered distributions by
$$\dot{\Delta}_j u\stackrel{\mathrm{def}}{=}\varphi(2^{-j}D)u\stackrel{\mathrm{def}}{=}{\mathcal F}^{-1}(\varphi(2^{-j}\cdot){\mathcal F} u).
$$
In order to ensure that
\begin{equation}\label{eq:decompo}
u=\sum_{j\in\Z}\dot{\Delta}_j u\quad\hbox{in}\quad\mathcal S'(\R^n),
\end{equation}
we restrict our attention to  those tempered distributions $u$ such that
\begin{equation}\label{eq:Sh}
\lim_{k\rightarrow-\infty}\|\dot S_ku\|_{L^\infty}=0,
\end{equation}
where $\dot S_k u$ stands for the low frequency cut-off defined by $\dot S_ku\stackrel{\mathrm{def}}{=}\chi(2^{-k}D)u$.
\begin{definition}\label{d:besov}
 For $s\in\R$, $1\le p\le\infty,$   the homogeneous Besov space $\dot B^s_{p,1}\stackrel{\mathrm{def}}{=}\dot B^{s}_{p,1}(\R^n)$ is  the
set of tempered distributions $u$ satisfying \eqref{eq:Sh} and
\begin{equation}\label{wqq1}
\|u\|_{\dot B^s_{p,1}}\stackrel{\mathrm{def}}{=}\sum_{j\in\Z} 2^{js}
\|\dot{\Delta}_j  u\|_{L^p}<\infty.
\end{equation}
\end{definition}
\begin{remark} For $s\leq \frac np$ (which is the only case we are concerned with in
this paper), $\dot B^s_{p,1}$ is a Banach space which coincides with
the completion for $\|\cdot\|_{\dot B^s_{p,1}}$ of the set $\mathcal S_0(\R^n)$ of Schwartz functions
with Fourier transform supported away from the origin.
\end{remark}

In this paper, we frequently use the so-called ''time-space" Besov spaces
or Chemin-Lerner space first introduced by Chemin and Lerner \cite{bcd}.
\begin{definition}
Let $s\in\mathbb{R}$ and $0<T\leq +\infty$. We define
\begin{equation}\label{wqq2}
\|u\|_{\widetilde{L}_{T}^{q}(\dot{B}_{p,1}^{s})}\stackrel{\mathrm{def}}{=}
\sum_{j\in\mathbb{Z}}
2^{js}\left(\int_{0}^{T}\|\dot{\Delta}_{j}u(t)\|_{L^p}^{q}dt\right)^{\frac{1}{q}}
\end{equation}
for  $ q,\,p \in [1,\infty)$ and with the standard modification for $p,\,q =\infty $.
\end{definition}

By Minkowski's inequality, we have the following inclusions between the
Chemin-Lerner space ${\widetilde{L}^\lambda_{T}(\dot{B}_{p,1}^s)}$ and the Bochner space ${{L}^\lambda_{T}(\dot{B}_{p,1}^s)}$:
\begin{align*}
\|u\|_{\widetilde{L}^\lambda_{T}(\dot{B}_{p,1}^s)}\le\|u\|_{L^\lambda_{T}(\dot{B}_{p,1}^s)}\hspace{0.5cm} \mathrm{if }\hspace{0.2cm}  \lambda\le r,\hspace{0.5cm}
\|u\|_{\widetilde{L}^\lambda_{T}(\dot{B}_{p,1}^s)}\ge\|u\|_{L^\lambda_{T}(\dot{B}_{p,1}^s)},\hspace{0.5cm} \mathrm{if }\hspace{0.2cm}  \lambda\ge r.
\end{align*}

Restricting the above norms \eqref{wqq1} and \eqref{wqq2} to the low or high
frequency parts of distributions will be crucial in our approach. For example, let us fix some integer $j_{0}$ (the value of which will
follow from the proof of the main theorem) and set$^1$\footnote{$^1$Note that for technical reasons, we need a small
overlap between low and high frequency.}
\begin{equation}\label{wqq3}
\| z^{\ell}\| _{\dot{B}_{p,1}^{s}} \stackrel{\mathrm{def}}{=} \sum_{j\leq
j_{0}}2^{js}\| \dot{\Delta}_{j}z\|_{L^{p}} \ \ \mbox{and} \ \ \|z^{h}\|_{\dot{B}_{p,1}^{s}}\stackrel{\mathrm{def}}{=} \sum_{j\geq j_{0}-2}2^{js}\| \dot{\Delta}_{j}z\| _{L^{p}},
\end{equation}
\begin{equation}\label{wqq4}
\|z^{\ell}\| _{\widetilde{L}_{T}^{\infty} (\dot{B}_{p,1}^{s})} \stackrel{\mathrm{def}}{=}
\sum_{j\leq j_{0}}2^{js}\|\dot{\Delta}_{j}z\|_{L_{T}^{\infty} (L^{p})} \ \
\mbox{and} \ \ \|z^{h}\| _{\widetilde{L}_{T}^{\infty} (\dot{B}_{p,1}^{s})}\stackrel{\mathrm{def}}{=} \sum_{j\geq j_{0}-2}2^{js}\| \dot{\Delta}_{j}z\|
_{L_{T}^{\infty} (L^{p})}.
\end{equation}

The following Bernstein's lemma will be repeatedly used throughout this paper.
\begin{lemma}\label{bernstein}
Let $\mathcal{B}$ be a ball and $\mathcal{C}$ a ring of $\mathbb{R}^n$. A constant $C$ exists so that for any positive real number $\lambda$, any
non-negative integer k, any smooth homogeneous function $\sigma$ of degree m, and any couple of real numbers $(p, q)$ with
$1\le p \le q\le\infty$, there hold
\begin{align*}
&&\mathrm{Supp} \,\hat{u}\subset\lambda \mathcal{B}\Rightarrow\sup_{|\alpha|=k}\|\partial^{\alpha}u\|_{L^q}\le C^{k+1}\lambda^{k+n(\frac1p-\frac1q)}\|u\|_{L^p},\\
&&\mathrm{Supp} \,\hat{u}\subset\lambda \mathcal{C}\Rightarrow C^{-k-1}\lambda^k\|u\|_{L^p}\le\sup_{|\alpha|=k}\|\partial^{\alpha}u\|_{L^p}
\le C^{k+1}\lambda^{k}\|u\|_{L^p},\\
&&\mathrm{Supp} \,\hat{u}\subset\lambda \mathcal{C}\Rightarrow \|\sigma(D)u\|_{L^q}\le C_{\sigma,m}\lambda^{m+n(\frac1p-\frac1q)}\|u\|_{L^p}.
\end{align*}
\end{lemma}

Next we  recall a few nonlinear estimates in Besov spaces which may be
obtained by means of paradifferential calculus.
Here, we recall the decomposition in the homogeneous context:
\begin{align}\label{bony}
uv=\dot{T}_uv+\dot{T}_vu+\dot{R}(u,v),
\end{align}
where
$$\dot{T}_uv\stackrel{\mathrm{def}}{=}\sum_{j\in Z}\dot{S}_{j-1}u\dot{\Delta}_jv, \hspace{0.5cm}\dot{R}(u,v)\stackrel{\mathrm{def}}{=}\sum_{j\in Z}
\dot{\Delta}_ju\widetilde{\dot{\Delta}}_jv,\quad\hbox{
and} \quad \widetilde{\dot{\Delta}}_jv\stackrel{\mathrm{def}}{=}\sum_{|j-j'|\le1}\dot{\Delta}_{j'}v.$$

The paraproduct $\dot{T}$ and the remainder $\dot{R}$ operators satisfy the following
continuous properties.

\begin{lemma}[\cite{bcd}]\label{fangji}
For all $s\in\mathbb{R}$, $\sigma\ge0$, and $1\leq p, p_1, p_2\leq\infty,$ the
paraproduct $\dot T$ is a bilinear, continuous operator from $\dot{B}_{p_1,1}^{-\sigma}\times \dot{B}_{p_2,1}^s$ to
$\dot{B}_{p,1}^{s-\sigma}$ with $\frac{1}{p}=\frac{1}{p_1}+\frac{1}{p_2}$. The remainder $\dot R$ is bilinear continuous from
$\dot{B}_{p_1, 1}^{s_1}\times \dot{B}_{p_2,1}^{s_2}$ to $
\dot{B}_{p,1}^{s_1+s_2}$ with
$s_1+s_2>0$, and $\frac{1}{p}=\frac{1}{p_1}+\frac{1}{p_2}$.
\end{lemma}

\begin{lemma}\label{jiao2}
Let   $ n=2,3$ and
\begin{equation*}
2\leq p \leq \min(4,{2n}/({n-2}))\quad\hbox{and, additionally, }\  p\not=4\ \hbox{ if }\ n=2.
\end{equation*} 
For any $(u,v)\in\dot{B}_{p,1}^{\frac {n}{p}-1}\cap\dot{B}_{p,1}^{\frac {n}{p}}(\R^n), $  there holds
\begin{align}\label{pengyou11}
&\|uv\|_{\dot{B}_{2,1}^{\f n2-1}}\lesssim(\|u\|_{\dot{B}_{p,1}^{\frac {n}{p}-1}}\|v\|_{\dot{B}_{p,1}^{\frac {n}{p}}}+\|u\|_{\dot{B}_{p,1}^{\frac {n}{p}}}\|v\|_{\dot{B}_{p,1}^{\frac {n}{p}-1}}).
\end{align}
\end{lemma}
\begin{proof}
According to Bony's decomposition, we can write $$uv=\dot{T}_uv+\dot{T}_vu+\dot{R}(u,v).$$
By Lemma \ref{fangji}, let $\frac  1{p*}=\frac  12-\frac1p$, we have
$$\|\dot{T}_uv+\dot{R}(u,v)\|_{\dot{B}_{2,1}^{\f n2-1}}\lesssim \|u\|_{\dot{B}_{p*,1}^{\frac {n}{p*}-1}}\|v\|_{\dot{B}_{p,1}^{\frac {n}{p}}}\lesssim\|u\|_{\dot{B}_{p,1}^{\frac {n}{p}-1}}\|v\|_{\dot{B}_{p,1}^{\frac {n}{p}}}.
$$
Similarly, one can get
$$\quad \|\dot{T}_vu\|_{\dot{B}_{2,1}^{\f n2-1}}\lesssim \|v\|_{\dot{B}_{p,1}^{\frac {n}{p}-1}}\|u\|_{\dot{B}_{p,1}^{\frac {n}{p}}}.$$
Thus, we complete the proof of this lemma.
\end{proof}

We also need the  following omitted proofs
 product law and commutator's estimates in Besov spaces.
\begin{lemma}\label{daishu}
Let $1\leq p, q\leq \infty$, $s_1\leq \frac{n}{q}$, $s_2\leq n\min\{\frac1p,\frac1q\}$, and $s_1+s_2>n\max\{0,\frac1p +\frac1q -1\}$. For $\forall (u,v)\in\dot{B}_{q,1}^{s_1}({\mathbb R} ^n)\times\dot{B}_{p,1}^{s_2}({\mathbb R} ^n)$, we have
\begin{align*}
\|uv\|_{\dot{B}_{p,1}^{s_1+s_2 -\frac{n}{q}}}\lesssim \|u\|_{\dot{B}_{q,1}^{s_1}}\|v\|_{\dot{B}_{p,1}^{s_2}}.
\end{align*}
\end{lemma}
\begin{lemma}(  \cite[Lemma 2.100]{bcd})\label{jiaohuanzi}
Let $1\leq p, q\leq \infty$, $-1-n\min\{\frac1p,1-\frac{1}{p}\}<s\leq 1+\frac np$.
For any $v\in \dot{B}_{p,1}^{s}(\R^n)$ and $\nabla u\in \dot{B}_{p,1}^{\frac{n}{p}}(\R^n)$ with $\div u=0$,
there holds
$$
\|[u\cdot \nabla ,\dot{\Delta}_j]v\|_{L^p}\lesssim d_j 2^{-js}\|\nabla u\|_{\dot{B}_{p,1}^{\frac{n}{p}}}\|v\|_{\dot{B}_{p,1}^{s}}.
$$
\end{lemma}

\begin{lemma}\label{xin}
Let   $ n=2,3$ and
\begin{equation*}
2\leq p \leq \min(4,{2n}/({n-2}))\quad\hbox{and, additionally, }\  p\not=4\ \hbox{ if }\ n=2.
\end{equation*}  For any   $v^\ell\in\dot{B}^{\frac{n}{2}-1}_{2,1}(\R^n)$, $v^h\in\dot{B}^{\f np-1}_{p,1}(\R^n),$ $\nabla u\in \dot{B}^{\f np}_{p,1}(\R^n)$ with $\div u = 0$,  there exists a constant $C$ such that
\begin{align}\label{xin2}
\sum_{j\le j_0}2^{(\frac{n}{2}-1)j}
\big\|[\ddj, u\cdot\nabla]v\big\|_{L^2}\le& C \big\|\nabla u\big\|_{\dot{B}^{\f np}_{p,1}}(\big\|v^\ell\big\|_{\dot{B}^{\f n2-1}_{2,1}}+\big\|v^h\big\|_{\dot{B}^{\f np-1}_{p,1}}).
\end{align}
\end{lemma}
\begin{proof}
Using the notion of para-products, we can easily write
$$[\dot{\Delta}_j,u\cdot\nabla]v\stackrel{\mathrm{def}}{=}I_j^1+I_j^2+I_j^3,$$ with
\begin{align*}
&I_j^1\stackrel{\mathrm{def}}{=}\sum_{|k-j|\le 2}[\dot{\Delta}_j,\dot{S}_{k-1}u\cdot\nabla]\dot{\Delta}_kv,\quad I_j^2\stackrel{\mathrm{def}}{=}\sum_{|k-j|\le 2}[\dot{\Delta}_j,\dot{\Delta}_ku\cdot\nabla]\dot{S}_{k-1}v,\nonumber\\
&I_j^3\stackrel{\mathrm{def}}{=}\sum_{k\ge j-1}[\dot{\Delta}_j,\dot{\Delta}_ku\cdot\nabla]\widetilde{\dot{\Delta}_k} v,\quad\quad \ \widetilde{\dot{\Delta}_k}=\dot{\Delta}_{k-1}+\dot{\Delta}_k+\dot{\Delta}_{k+1}.
\end{align*}%
From the definition of Bony's decomposition, one can write $I_j^1$ into
\begin{align*}
I_j^1=&\sum_{|k-j|\le2}[\dot{\Delta}_j,\dot{S}_{k-1}{{u}_m} ]\partial_m\dot{\Delta}_{k}v\nonumber\\
=&2^{jn}\sum_{|k-j|\leq2}\int_{\R^n}h(2^jy)
(\dot{S}_{k-1}{u}_m(x-y)-\dot{S}_{k-1}{u}_m(x))\partial_m\dot{\Delta}_{k}v(x-y)dy
\nonumber\\=&-2^{jn}\sum_{|k-j|\leq2}\int_{\R^n}h(2^jy)
\Big(\int_0^1y\cdot\nabla \dot{S}_{k-1}{u}_m(x-\tau y)\,d\tau\Big) \partial_m\dot{\Delta}_{k}v(x-y)dy
\end{align*}
from which and the H\"older inequality, we have
\begin{align}\label{tianchang}
\|I_j^1\|_{L^2}\lesssim&\sum_{|k-j|\le 2}2^{-j}\|\nabla \dot{S}_{k-1}u\|_{L^\infty}\|\nabla\dot{\Delta}_kv\|_{L^2}\nonumber\\
\lesssim&\sum_{|k-j|\le 2}2^{k-j}\|\nabla \dot{S}_{k-1}u\|_{L^\infty}\|\dot{\Delta}_kv\|_{L^2}.
\end{align}
As there is a small
overlap between low and high frequency in the definition of \eqref{wqq3}, we can further deduce from \eqref{tianchang}
that
\begin{align}\label{yuyi1}
\sum_{j\le j_0}2^{(\frac{n}{2}-1)j}\|I_j^1\|_{L^2}
\le& C\big\|\nabla u\big\|_{L^\infty}\big\|v^\ell\big\|_{\dot{B}^{\frac{n}{2}-1}_{2,1}}.
\end{align}

Let us turn to the second term $I_j^2$.
Using the fact that the support of $\dot{\Delta}_j(\dot{\Delta}_ku\cdot\nabla\dot{S}_{k-1}v)$ is restricted in an
annulus, we can get similarly to $I_j^1 $ that
\begin{align*}
\|I_j^2\|_{L^2}\lesssim&\sum_{|k-j|\le 2}2^{-j}\|\nabla \dot{S}_{k-1}v\|_{L^{\frac{2p}{p-2}}}\|\nabla\dot{\Delta}_ku\|_{L^p}\nonumber\\
\lesssim&\sum_{|k-j|\le 2}2^{-j}(\sum_{k'\le k-2}2^{k'}2^{\frac{nk'}{p}}\| \dot{\Delta}_{k'}v\|_{L^2})\|\nabla\dot{\Delta}_ku\|_{L^p}\nonumber\\
\lesssim&\sum_{|k-j|\le 2}(2^{(1-\frac{n}{2})k}d_k\big\|v^\ell\big\|_{\dot{B}^{\f n2-1}_{2,1}}(2^{\frac{nk}{p}}\|\nabla\dot{\Delta}_ku\|_{L^p}),
\end{align*}
which give rise to
\begin{align}\label{yuyi11}
\sum_{j\le j_0}2^{(\frac{n}{2}-1)j}\|I_j^2\|_{L^2}
\le& C\big\|\nabla u\big\|_{\dot{B}^{\f np}_{p,1}}\big\|v^\ell\big\|_{\dot{B}^{\frac{n}{2}-1}_{2,1}}.
\end{align}

It is much more involved to handle the remainder term $I_j^3$. We split it into two terms:
high frequency and low frequency
\begin{align}\label{geuq}
I_j^3=&\sum_{k\ge j-1}[\dot{\Delta}_j,\dot{\Delta}_ku\cdot\nabla]\widetilde{\dot{\Delta}_k} v\nonumber\\
=&\sum_{ j-1\le k\le j}[\dot{\Delta}_j,\dot{\Delta}_ku\cdot\nabla]\widetilde{\dot{\Delta}_k} v+\sum_{ k> j}[\dot{\Delta}_j,\dot{\Delta}_ku\cdot\nabla]\widetilde{\dot{\Delta}_k} v.
\end{align}
Exact the same line as $I_j^1$, we can get
\begin{align}\label{yuyi12}
\sum_{j\le j_0}2^{(\frac{n}{2}-1)j}\|\sum_{ j-1\le k\le j}[\dot{\Delta}_j,\dot{\Delta}_ku\cdot\nabla]\widetilde{\dot{\Delta}_k} v\|_{L^2}
\le& C\big\|\nabla u\big\|_{L^\infty}\big\|v^\ell\big\|_{\dot{B}^{\frac{n}{2}-1}_{2,1}}.
\end{align}
Due to lack of  quasi-orthogonality, we divide the second term on the right hand side of \eqref{geuq} into two terms:
\begin{align*}
\sum_{k> j}[\dot{\Delta}_j,\dot{\Delta}_ku\cdot\nabla]\widetilde{\dot{\Delta}_k} v
=&\sum_{k> j}\dot{\Delta}_j(\dot{\Delta}_ku\cdot\nabla\widetilde{\dot{\Delta}_k} v)+\sum_{ j<k,|k-j|\le3}\dot{\Delta}_ku\cdot\nabla\dot{\Delta}_j\widetilde{\dot{\Delta}_k} v\\
\stackrel{\mathrm{def}}{=}&I_j^{3,1}+I_j^{3,2}.
\end{align*}

To bound  $I_j^{3,1}$, we need to further write
\begin{align}\label{yuyi13}
\|I_j^{3,1}\|_{L^2}=\sum_{{j}<{k}\le j_0}\|I_j^{3,1}\|_{L^2}+\sum_{{j}\le j_0< k}\|I_j^{3,1}\|_{L^2}.
\end{align}
Using the condition $\div u=0$ and the H$\mathrm{\ddot{o}}$lder inequality gives
\begin{align*}
\sum_{{j}<{k}\le j_0}\|I_j^{3,1}\|_{L^2}\lesssim&\sum_{{j}<{k}< j_0}\|\dot{\Delta}_j\partial_m(\dot{\Delta}_ku_m\cdot\widetilde{\dot{\Delta}_k} v)\|_{L^2}\nonumber\\
\lesssim&\sum_{{j}<{k}< j_0}2^j\|\dot{\Delta}_ku\cdot\widetilde{\dot{\Delta}_k} v\|_{L^2}\nonumber\\
\lesssim&\sum_{{j}<{k}< j_0}2^{j-k}\|\dot{\Delta}_k\nabla u\|_{L^\infty}\|\widetilde{\dot{\Delta}_k} v\|_{L^2}\nonumber\\
\lesssim&\|\nabla u\|_{L^\infty}\sum_{{j}<{k}< j_0}2^{j-k}\|\widetilde{\dot{\Delta}_k} v\|_{L^2},
\end{align*}
which implies
\begin{align}\label{yuyi14}
\sum_{{j}<{k}\le j_0}2^{(\frac{n}{2}-1)j}\|I_j^{3,1}\|_{L^2}
\le& C\big\|\nabla u\big\|_{L^\infty}\big\|v^\ell\big\|_{\dot{B}^{\frac{n}{2}-1}_{2,1}}.
\end{align}
Similarly, the second term in \eqref{yuyi13} can be estimated as follow:
\begin{align}\label{ref}
\sum_{{j}\le j_0< k}2^{(\frac{n}{2}-1)j}\|I_j^{3,1}\|_{L^2}\lesssim&\sum_{{j}\le j_0< k}2^{(\frac{n}{2}-1)j}\|\dot{\Delta}_j\partial_m(\dot{\Delta}_ku_m\cdot\widetilde{\dot{\Delta}_k} v)\|_{L^2}\nonumber\\
\lesssim&\sum_{{j}\le j_0< k}2^{(\frac{n}{2}-1)j}2^j\|\dot{\Delta}_ku\cdot\widetilde{\dot{\Delta}_k} v\|_{L^2}\nonumber\\
\lesssim&\sum_{{j}\le j_0< k}2^{(\frac{n}{2}-1)j}2^{j-k}\|\dot{\Delta}_k\nabla u\|_{L^{\frac{2p}{p-2}}}\|\widetilde{\dot{\Delta}_k} v\|_{L^p}\nonumber\\
\lesssim&\sum_{{j}\le j_0< k}2^{(\frac{n}{2}-1)j}2^{j-k}2^{(\frac{1}{p}-\frac{p-2}{2p})nk}\|\dot{\Delta}_k\nabla u\|_{L^p}\|\widetilde{\dot{\Delta}_k} v\|_{L^p}\nonumber\\
\lesssim&\|\nabla u\|_{\dot{B}^{\frac{n}{p}}_{p,1}}\|v^h\|_{\dot{B}^{\frac{n}{p}-1}_{p,1}}.
\end{align}

In virtue of the embedding relation $\dot{B}^{\frac{n}{p}}_{p,1}(\R^n)\hookrightarrow L^\infty(\R^n)$,
 we get from \eqref{yuyi14} and \eqref{ref} that
\begin{align}\label{yuyi15}
\sum_{{j}\le j_0}2^{(\frac{n}{2}-1)j}\|I_j^{3,1}\|_{L^2}
\le& C\|\nabla u\|_{\dot{B}^{\frac{n}{p}}_{p,1}}(\|v^h\|_{\dot{B}^{\frac{n}{p}-1}_{p,1}}+\big\|v^\ell\big\|_{\dot{B}^{\frac{n}{2}-1}_{2,1}}).
\end{align}
Thanks to  Lemma \ref{bernstein}, we have
\begin{align}\label{yuyi16}
\sum_{{j}\le j_0}2^{(\frac{n}{2}-1)j}\|I_j^{3,2}\|_{L^2}
\lesssim&\sum_{{j}\le j_0}2^{(\frac{n}{2}-1)j}\sum_{j<k,|k-j|\le3}\|\dot{\Delta}_ku\cdot\nabla\dot{\Delta}_j\widetilde{\dot{\Delta}_k} v\|_{L^2}\nonumber\\
\lesssim&\sum_{{j}\le j_0}2^{(\frac{n}{2}-1)j}\sum_{|k-j|\le3}2^{j-k}\|\nabla \dot{\Delta}_ku\|_{L^\infty}\|\dot{\Delta}_j\widetilde{\dot{\Delta}_k} v\|_{L^2}\nonumber\\
\lesssim&\big\|\nabla u\big\|_{L^\infty}\big\|v^\ell\big\|_{\dot{B}^{\frac{n}{2}-1}_{2,1}}.
\end{align}

Together with \eqref{yuyi12}, \eqref{yuyi15}, \eqref{yuyi16}, we get from \eqref{geuq} that
\begin{align}\label{yuyi18}
\sum_{{j}\le j_0}2^{(\frac{n}{2}-1)j}\|I_j^{3}\|_{L^2}
\le& C\|\nabla u\|_{\dot{B}^{\frac{n}{p}}_{p,1}}(\|v^h\|_{\dot{B}^{\frac{n}{p}-1}_{p,1}}+\big\|v^\ell\big\|_{\dot{B}^{\frac{n}{2}-1}_{2,1}}).
\end{align}
Thus, the estimate \eqref{xin2} can be obtained from the combinations of \eqref{yuyi1}, \eqref{yuyi11}, \eqref{yuyi18}.

Consequently, we complete the proof of the lemma.
\end{proof}

\begin{corollary}\label{xinqing}
Under the assumption of Lemma \ref{xin}, let $A(D)$ be a zero-order Fourier multiplier, by the same processes as the proof of Lemma \ref{xin}, we can get the following two estimates hold:
\begin{align*}
\sum_{{j}\le j_0}2^{(\frac{n}{2}-1)j}
\big\|[\ddj A(D), u\cdot\nabla]v)\big\|_{L^2}\le& C \big\|\nabla u\big\|_{\dot{B}^{\f np}_{p,1}}(\big\|v^\ell\big\|_{\dot{B}^{\f n2-1}_{2,1}}+\big\|v^h\big\|_{\dot{B}^{\f np-1}_{p,1}}),\\
\sum_{{j}\le j_0}2^{(\frac{n}{2}-1)j}
\big\|[\ddj, u\cdot\nabla]A(D)v)\big\|_{L^2}\le& C \big\|\nabla u\big\|_{\dot{B}^{\f np}_{p,1}}(\big\|v^\ell\big\|_{\dot{B}^{\f n2-1}_{2,1}}+\big\|v^h\big\|_{\dot{B}^{\f np-1}_{p,1}})
.
\end{align*}
\end{corollary}

\begin{lemma}\label{dahai}  (  \cite[Lemma 6.1]{helingbing})
Let $A(D)$ be a zero-order Fourier multiplier.
Let $j_0\in\Z,$ $s<1,$ $\sigma\in\R$,  $1\leq p_1,\ p_2\leq\infty$ and  $\frac  1p=\frac  1{p_1}+\frac  1{p_2}.$
Then there exists a constant $C$ depending only on $j_0$ and
on the regularity parameters such that
$$
\big\|[\dot S_{j_0} A(D), T_u ]v\big\|_{\dot B^{\sigma+s}_{p,1}}
\le C\|\nabla u \|_{\dot B^{s-1}_{p_1,1} }\|v\|_{\dot B^\sigma_{p_2,1}}
$$
and, for  $s=1,$
$$
\big\|[\dot S_{j_0} A(D), T_u ]v\big\|_{\dot B^{\sigma+1}_{p,1}}
\leq C\|\nabla u \|_{L^{p_1}}\|v\|_{\dot B^\sigma_{p_2,1}}.
$$
\end{lemma}

\setcounter{section}{3}
\setcounter{equation}{0}
\section*{\Large\bf 3. The proof of the Theorem \ref{zhuyaodingli}}
According to the local well-posedness obtained by \cite{chemin}, \cite{chenqionglei},  we can deduce similarly that  there exists a positive time $T$ so that the system \eqref{m} has a uniqueness local solution $(u,\tau)$ on $[0,T^\ast)$ such that for  any $T<T^\ast$
\begin{align}\label{tian}
&u^\ell\in C_b([0,T );{\dot{B}}_{2,1}^{\frac {n}{2}-1}(\R^n))\cap L^{1}
([0,T];{\dot{B}}_{2,1}^{\frac n2+1}(\R^n)),\nonumber\\
&
\tau^\ell\in C_b([0,T );\dot{B}_{2,1}^{\frac n2-1}(\R^n)), \quad(\Lambda^{-1}\p\div\tau)^\ell\in L^{1}
([0,T];{\dot{B}}_{2,1}^{\frac n2+1}(\R^n)),\nonumber\\
&u^h\in C_b([0,T );{\dot{B}}_{p,1}^{\frac {n}{p}-1}(\R^n))\cap L^{1}
([0,T];{\dot{B}}_{p,1}^{\frac np+1}(\R^n)),\nonumber\\
&
\tau^h\in C_b([0,T );\dot{B}_{p,1}^{\frac np}(\R^n)), \quad (\Lambda^{-1}\p\div\tau)^h\in L^{1}
([0,T];{\dot{B}}_{p,1}^{\frac np}(\R^n)).
\end{align}
We denote $T^\ast$ to be the largest possible time such that there holds \eqref{tian}. Then, the proof of Theorem \ref{zhuyaodingli} is reduced to show that $T^\ast=\infty$ under the assumption of \eqref{smallness2}.
In order to do so, we need to make a priori estimates for the smooth solution of system \eqref{m}.

\bigskip
\subsection{The low frequency estimates of the solutions }

Applying   $\dot{\Delta}_j\p$ to the second equation in  \eqref{m} and using a standard commutator's process give
\begin{align}\label{caihong1}
\partial_t \dot{\Delta}_ju+ u\cdot\nabla \dot{\Delta}_ju-\Delta \dot{\Delta}_ju-\dot{\Delta}_j\p\div \tau=[u\cdot\nabla, \dot{\Delta}_j\p]u.
\end{align}
Similarly, from the first equation in  \eqref{m}, we have
\begin{align}\label{caihong2}
\partial_t \dot{\Delta}_j\tau+ u\cdot \nabla \dot{\Delta}_j\tau   +\dot{\Delta}_j F(\tau, \nabla u)  -  \dot{\Delta}_jD (u)=[u\cdot\nabla, \dot{\Delta}_j]\tau.
\end{align}
Taking $L^2$ inner product with $\dot{\Delta}_ju$  on both hand side of \eqref{caihong1} and using
the fact that $
\langle u\cdot\nabla \dot{\Delta}_ju,\dot{\Delta}_ju\rangle=0$ give
\begin{align}\label{caihong3}
\frac{1}{2}\frac{d}{dt}\|\dot{\Delta}_ju\|_{L^2}^2+c_1\|\nabla \dot{\Delta}_ju\|_{L^2}^2=\langle\dot{\Delta}_j\p\div \tau, \dot{\Delta}_ju\rangle+\langle[u\cdot\nabla, \dot{\Delta}_j\p] u,\dot{\Delta}_j u\rangle.
\end{align}
Similarly,
taking $L^2$ inner product with $\dot{\Delta}_j\tau$  on both hand side of \eqref{caihong2} and using
the fact that $
\langle u\cdot\nabla \dot{\Delta}_j\tau,\dot{\Delta}_j\tau\rangle=0$,
we can get
\begin{align}\label{caihong5}
\frac{1}{2}\frac{d}{dt}\|\dot{\Delta}_j\tau\|_{L^2}^2=\langle \dot{\Delta}_jD (u), \dot{\Delta}_j\tau\rangle+\langle[u\cdot\nabla, \dot{\Delta}_j] \tau,\dot{\Delta}_j \tau\rangle
-\langle\dot{\Delta}_j F(\tau, \nabla u),\dot{\Delta}_j\tau\rangle.
\end{align}
It's not difficult to check
$$\langle\dot{\Delta}_j\p\div \tau, \dot{\Delta}_ju\rangle+\langle \dot{\Delta}_jD (u), \dot{\Delta}_j\tau\rangle=0.$$

Thus, summing up \eqref{caihong3}, \eqref{caihong5} and using the above fact we have
\begin{align}\label{caihong6}
&\frac{1}{2}\frac{d}{dt}(\|\dot{\Delta}_ju\|_{L^2}^2+\|\dot{\Delta}_j\tau\|_{L^2}^2)+c_12^{2j}\| \dot{\Delta}_ju\|_{L^2}^2\nonumber\\
&\quad\lesssim \left|\langle[u\cdot\nabla, \dot{\Delta}_j\p] u,\dot{\Delta}_j u\rangle\right|+\left|\langle[u\cdot\nabla, \dot{\Delta}_j] \tau,\dot{\Delta}_j \tau\rangle\right|
+\left|\langle\dot{\Delta}_j F(\tau, \nabla u),\dot{\Delta}_j\tau\rangle\right|
\end{align}
in which we have used the following Bernstein's inequality:
there exists a positive constant $c_1$ so that
$$
-\int_{\R^n}\Delta\dot{\Delta}_ju\cdot\dot{\Delta}_ju\,dx\ge c_12^{2j}\|\dot{\Delta}_ju\|_{L^2}^2.
$$

Due to lack of full dissipation for stress tensor $\tau$ in \eqref{m}, thus, we have to give up the  dissipation for $u$ also at present.
In the following, we will  get back the  full dissipation of velocity
and the partial dissipation of stress tensor by introducing a new   quantity.

Employing the H${\mathrm{\ddot{o}}}$lder inequality to \eqref{caihong6},  integrating  the resultant inequality from $0$ to $t$,  and multiplying   by $2^{j(\frac{n}{2}-1)}$, we can get by summing up about $j\le j_0$ that
\begin{align}\label{caihong7}
\|(u^\ell,\tau^\ell)\|_{\widetilde{L}_t^{\infty}(\dot{B}_{2,1}^{\frac n2-1})}
\lesssim&\|(u_0^\ell,\tau_0^\ell)\|_{\dot{B}_{2,1}^{\frac {n}{2}-1}}+{\sum_{j\le j_0}2^{(\frac n2-1)j}\|[u\cdot\nabla, \dot{\Delta}_j\p] u\|_{L^1_t(L^2)}}\nonumber\\
&+
{\sum_{j\le j_0}2^{(\frac n2-1)j}\|[u\cdot\nabla, \dot{\Delta}_j] \tau\|_{L^1_t(L^2)}}+\int^t_0\|(F(\tau, \nabla u))^\ell\|_{\dot{B}_{2,1}^{\frac {n}{2}-1}}\,ds.
\end{align}
It follows from  Lemma \ref{xin} that
\begin{align}\label{danian6-1}
&\sum_{j\le j_0}2^{(\frac n2-1)j}\|[u\cdot\nabla, \dot{\Delta}_j\p] u\|_{L^1_t(L^2)}\nonumber\\
&\quad\lesssim\int^t_0\|\nabla u\|_{\dot{B}_{p,1}^{\frac {n}{p}}}\big(\|u^\ell\|_{\dot{B}_{2,1}^{\frac n2-1}}+\|u^h\|_{\dot{B}_{p,1}^{\frac np-1}})\,ds
\nonumber\\
&\quad\lesssim\int^t_0\big(\|u^\ell\|_{\dot{B}_{2,1}^{\frac n2-1}}+\|u^h\|_{\dot{B}_{p,1}^{\frac np-1}})(\|u^\ell\|_{\dot{B}_{2,1}^{\frac n2+1}}+\|u^h\|_{\dot{B}_{p,1}^{\frac np+1}})\,ds
\end{align}
and
\begin{align}\label{danian6}
&
\sum_{j\le j_0}2^{(\frac n2-1)j}\|[u\cdot\nabla, \dot{\Delta}_j]\tau\|_{L^1_t(L^2)}\nonumber\\
&\quad\lesssim\int^t_0\|\nabla u\|_{\dot{B}_{p,1}^{\frac {n}{p}}}\big(\|\tau^\ell\|_{\dot{B}_{2,1}^{\frac n2-1}}+\|\tau^h\|_{\dot{B}_{p,1}^{\frac np}})\,ds
\nonumber\\
&\quad\lesssim\int^t_0\big(\|\tau^\ell\|_{\dot{B}_{2,1}^{\frac n2-1}}+\|\tau^h\|_{\dot{B}_{p,1}^{\frac np}})(\|u^\ell\|_{\dot{B}_{2,1}^{\frac n2+1}}+\|u^h\|_{\dot{B}_{p,1}^{\frac np+1}})\,ds.
\end{align}
In order to estimate the last term in \eqref{caihong7}, we first use the Bony decomposition to write
\begin{align}\label{D9}
\dot S_{j_0+1} (\tau\nabla u)=\dot S_{j_0+1} \bigl(\dot{T}_{\tau} \nabla u+ \dot{R}(\tau,\nabla u)\bigr)+\dot{T}_{\nabla u}\dot S_{j_0+1}  \tau
+[\dot S_{j_0+1} ,T_{\nabla u}]\tau.
\end{align}
By virtue of Lemma \ref{fangji}, we obtain
\begin{align}\label{D11-1}
\|\dot{T}_{\nabla u}\dot S_{j_0+1}  \tau\|_{\dot{B}_{2,1}^{\frac{n}{2}-1}}\lesssim\|\nabla u\|_{L^\infty}\|\tau^\ell\|_{\dot{B}_{2,1}^{\frac{n}{2}-1}}
\lesssim&\|\tau^\ell\|_{\dot{B}_{2,1}^{\frac{n}{2}-1}}\|\nabla u\|_{\dot{B}_{p,1}^{\frac {n}{p}}},
\end{align}
and for $\frac  1{p*}=\frac  12-\frac1p$
\begin{align}
\|\dot S_{j_0+1} \bigl(T_{\tau} \nabla u+ \dot{R}(\tau,\nabla u)\bigr)\|_{\dot{B}_{2,1}^{\frac{n}{2}-1}}\lesssim\|\tau\|_{\dot{B}_{p^*,1}^{-1+\frac {n}{p^*}}}\|\nabla u\|_{\dot{B}_{p,1}^{\frac {n}{p}}}
\lesssim&\|\tau\|_{\dot{B}_{p,1}^{-1+\frac {n}{p}}}\|\nabla u\|_{\dot{B}_{p,1}^{\frac {n}{p}}}.
\end{align}
By Lemma \ref{dahai},   we have
\begin{align}\label{D11}
\|[\dot S_{j_0+1} ,\dot{T}_{\nabla u}]\tau\|_{\dot{B}_{2,1}^{\frac{n}{2}-1}}\lesssim\|\nabla^2 u\|_{\dot{B}_{p^*,1}^{\frac {n}{p^*}-1}}\|\tau\|_{\dot{B}_{p,1}^{\frac {n}{p}-1}}
\lesssim&\|\tau\|_{\dot{B}_{p,1}^{\frac {n}{p}-1}}\|\nabla u\|_{\dot{B}_{p,1}^{\frac {n}{p}}}.
\end{align}

Combining with \eqref{D9}--\eqref{D11} implies
\begin{align}\label{danian7}
\int^t_0\|(F(\tau, \nabla u))^\ell\|_{\dot{B}_{2,1}^{\frac {n}{2}-1}}\,ds
\lesssim&\int^t_0\|\nabla u\|_{\dot{B}_{p,1}^{\frac {n}{p}}}(\|\tau^\ell\|_{\dot{B}_{2,1}^{\frac n2-1}}+\|\tau^h\|_{\dot{B}_{p,1}^{\frac np}})\,ds\nonumber\\
\lesssim&\int^t_0
(\|\tau^\ell\|_{\dot{B}_{2,1}^{\frac n2-1}}+\|\tau^h\|_{\dot{B}_{p,1}^{\frac np}})(\|u^\ell\|_{\dot{B}_{2,1}^{\frac n2+1}}+\|u^h\|_{\dot{B}_{p,1}^{\frac np+1}})\,ds.
\end{align}

Taking estimates \eqref{danian6} and \eqref{danian7} into \eqref{caihong7} gives
\begin{align}\label{nian8}
\|(u^\ell,\tau^\ell)\|&_{\widetilde{L}_t^{\infty}(\dot{B}_{2,1}^{\frac n2-1})}
\lesssim\|(u_0^\ell,\tau_0^\ell)\|_{\dot{B}_{2,1}^{\frac {n}{2}-1}}\nonumber\\
&+\int^t_0\big(\|(u^\ell,\tau^\ell)\|_{\dot{B}_{2,1}^{\frac n2-1}}+\|u^h\|_{\dot{B}_{p,1}^{\frac np-1}}+\|\tau^h\|_{\dot{B}_{p,1}^{\frac np}})(\|u^\ell\|_{\dot{B}_{2,1}^{\frac n2+1}}+\|u^h\|_{\dot{B}_{p,1}^{\frac np+1}})\,ds.
\end{align}

In the above low frequency arguments,  we do not get any integration in time for $u, \tau.$ Next,
we  shall use the special structure of
\eqref{m} to obtain the smoothing effect of $u$ and partial smoothing effect of $\tau$.

Applying project operator $\p$ on both hand side of the first two equation in \eqref{m} gives
\begin{eqnarray}\label{G1}
\left\{\begin{aligned}
&\partial_t  u +\p (u\cdot \nabla u)-\Delta  u-\p \div \tau=0,\\
&\partial_t \p \div \tau+\p \div(u\cdot \nabla \tau) -\Delta u+\p \div(F(\tau, \nabla u))=0.
\end{aligned}\right.
\end{eqnarray}
 Define $$\aa=\Lambda^{-1}\p\div\tau\quad\hbox{ and }\quad w=\Lambda \aa-u,$$ we can get by a simple computation from \eqref{G1} that
\begin{eqnarray}\label{G5}
\left\{\begin{aligned}
&\partial_t \aa+ u\cdot \nabla\aa+\Lambda u=f ,\\
&\partial_t  u + u\cdot \nabla u-\Delta  u-\Lambda \aa=g,\\
&\partial_t w+ u\cdot \nabla w+\La\aa=G,
\end{aligned}\right.
\end{eqnarray}
in which
\begin{align*}
&f\stackrel{\mathrm{def}}{=}- [\Lambda^{-1}\p\div,u\cdot \nabla ]\tau-\Lambda^{-1}\p \div(F(\tau, \nabla u)),\\
&g\stackrel{\mathrm{def}}{=}-[\p ,u\cdot \nabla] u,\quad\quad
G\stackrel{\mathrm{def}}{=}-[\La, u\cdot \nabla]\aa+\La f-g.
\end{align*}

As discussed in the first section, we will set our energy estimates  about \eqref{G5} in low frequency and high frequency respectively.
Applying $\ddj $ to the first equation in \eqref{G5} gives
\begin{align}\label{G6}
\partial_t \ddj \aa+u\cdot \nabla\ddj\aa+\ddj\La u=-[\ddj,u\cdot \nabla]\aa+\ddj f.
\end{align}
Taking $L^2$ inner product of $\ddj \aa$ with   \eqref{G6} and using integrating by parts, we obtain
\begin{align}\label{G7}
&\f12\f{d}{dt}\|{\ddj \aa}\|_{L^2}^2+\int_{\R^n}{\ddj \La\aa} \cdot\ddj u \,dx=-\int_{\R^n} [\ddj,u\cdot \nabla]\aa\cdot{\ddj \aa}\,dx+\int_{\R^n} \ddj f \cdot{\ddj \aa}\,dx.
\end{align}
Similarly, we have
\begin{align}
\f12\f{d}{dt}\|{\ddj u}\|_{L^2}^2+   2^{2j} \|{\ddj u}\|_{L^2}^2-&\int_{\R^n} \ddj \La \aa\cdot \ddj u\,dx\nonumber\\
&=-\int_{\R^n} [\ddj,u\cdot \nabla]u\cdot{\ddj u}\,dx+\int_{\R^n} \ddj g \cdot{\ddj u}\,dx,\label{G8}\\
\f12\f{d}{dt}\|{\ddj w}\|_{L^2}^2+\|{\ddj \La \aa}\|_{L^2}^2-&\int_{\R^n}\ddj \La \aa\cdot\ddj u \,dx\nonumber\\
&=-\int_{\R^n} [\ddj,u\cdot \nabla]w\cdot{\ddj w}\,dx+\int_{\R^n} \ddj G \cdot{\ddj w}\,dx,\label{G9}
\end{align}
in which we have used the following fact:
\begin{align*}
\int_{\R^n} \ddj \La \aa\cdot \ddj w\,dx=&\int_{\R^n}\ddj \La \aa\cdot(\ddj\Lambda \aa-\ddj u) \,dx\\
=&\|{\ddj \La \aa}\|_{L^2}^2-\int_{\R^n}\ddj \La \aa\cdot\ddj u \,dx.
\end{align*}

Let $0<\eta<1$  be a small constant which will be determined later on. Summing up \eqref{G7}--\eqref{G9} and using the H\"older inequality and Berntein's lemma, we have
\begin{align}\label{G12}
&\f12\f{d}{dt}(\|{\ddj \aa}\|_{L^2}^2+(1-\eta)\|{\ddj u}\|_{L^2}^2+\eta\|{\ddj w}\|_{L^2}^2)
+ (1-\eta)2^{2j} \|{\ddj u}\|_{L^2}^2+\eta2^{2j}\|{\ddj \aa}\|_{L^2}^2
\nonumber\\
&\quad\lesssim\|{\ddj \aa}\|_{L^2}(\|[\ddj,u\cdot \nabla]\aa\|_{L^2}+\|{\ddj f}\|_{L^2})+
\|{\ddj u}\|_{L^2}(\|[\ddj,u\cdot \nabla]u\|_{L^2}+\|{\ddj g}\|_{L^2})\nonumber\\
&\quad\quad+\|{\ddj w}\|_{L^2}(\|[\ddj,u\cdot \nabla]w\|_{L^2}+\|{\ddj G}\|_{L^2}).
\end{align}
For any $j\le j_0$, we can find an $\eta>0$ small enough such that
\begin{align}\label{G13}
\|{\ddj \aa}\|_{L^2}^2+(1-\eta)\|{\ddj u}\|_{L^2}^2+\eta\|{\ddj w}\|_{L^2}^2
\ge \frac{1}{C}(\|{\ddj \aa}\|_{L^2}^2+\|{\ddj u}\|_{L^2}^2).
\end{align}
From \eqref{G12}, one can deduce that
\begin{align}\label{G14}
&\f12\f{d}{dt}(\|{\ddj \aa}\|_{L^2}^2+\|{\ddj u}\|_{L^2}^2)
+ 2^{2j}( \|{\ddj  \aa}\|_{L^2}^2+\|{\ddj u}\|_{L^2}^2)
\nonumber\\
&\quad\lesssim\|{\ddj \aa}\|_{L^2}(\|[\ddj,u\cdot \nabla]\aa\|_{L^2}+\|{\ddj f}\|_{L^2})+
\|{\ddj u}\|_{L^2}(\|[\ddj,u\cdot \nabla]u\|_{L^2}+\|{\ddj g}\|_{L^2})\nonumber\\
&\quad\quad+\|{\ddj w}\|_{L^2}(\|[\ddj,u\cdot \nabla]w\|_{L^2}+\|{\ddj G}\|_{L^2}).
\end{align}

By the definition of the Besov space, we can further get
\begin{align}\label{G1125}
&\|(u^\ell,\aa^\ell)\|_{\widetilde{L}_t^{\infty}(\dot{B}_{2,1}^{\frac n2-1})}
+\int^t_0\| (u^\ell,\aa^\ell)\|_{\dot{B}_{2,1}^{\frac n2+1}}\,ds\nonumber\\
&\quad\lesssim\|(u_0^\ell,\aa_0^\ell)\|_{\dot{B}_{2,1}^{\frac {n}{2}-1}}
+
\int^t_0\|(f,G,g)^\ell\|_{\dot{B}_{2,1}^{\frac n2-1}}\,ds\nonumber\\
&\quad\quad+\int^t_0\sum_{j\le j_0}2^{(\frac n2-1)j}(\|[\ddj,u\cdot \nabla]\aa\|_{L^2}+\|[\ddj,u\cdot \nabla]u\|_{L^2}+\|[\ddj,u\cdot \nabla]\La\aa\|_{L^2})\,ds\nonumber\\
&\quad\lesssim\|(u_0^\ell,\aa_0^\ell)\|_{\dot{B}_{2,1}^{\frac {n}{2}-1}}
+
\int^t_0(\|f^\ell\|_{\dot{B}_{2,1}^{\frac n2-1}}+\|g^\ell\|_{\dot{B}_{2,1}^{\frac n2-1}}+\|([\La, u\cdot \nabla]\aa)^\ell\|_{\dot{B}_{2,1}^{\frac n2-1}})\,ds\nonumber\\
&\quad\quad+\int^t_0\sum_{j\le j_0}2^{(\frac n2-1)j}(\|[\ddj,u\cdot \nabla]\aa\|_{L^2}+\|[\ddj,u\cdot \nabla]u\|_{L^2}+\|[\ddj,u\cdot \nabla]\La\aa\|_{L^2})\,ds.
\end{align}

Next, we give the estimates to the terms in the righthand side of the above inequality.

\noindent A simple computation implies
\begin{align}\label{dfe}
\ddj([\Lambda^{-1}\p\div,u\cdot \nabla ]\tau)
=&\ddj(\Lambda^{-1}\p\div(u\cdot \nabla \tau)-\ddj(u\cdot \nabla \Lambda^{-1}\p\div\tau)\nonumber\\
=&[\ddj\Lambda^{-1}\p\div,u\cdot \nabla ]\tau-[\ddj,u\cdot \nabla]\Lambda^{-1}\p\div\tau.
\end{align}
As the Fourier multiplier $\Lambda^{-1} \p\div$ is of degree 0, thus, from \eqref{dfe} and Corollary \ref{xinqing}, we have
\begin{align}\label{G26+99}
&\| ([\Lambda^{-1}\p\div,u\cdot \nabla ]\tau)^\ell\|_{\dot{B}_{2,1}^{\frac n2-1}}\nonumber\\
&\quad\lesssim(\|\tau^\ell\|_{\dot{B}_{2,1}^{\frac n2-1}}+\|\tau^h\|_{\dot{B}_{p,1}^{\frac np-1}})\|\nabla u\|_{\dot{B}_{p,1}^{\frac np}}\nonumber\\
&\quad\lesssim(\|\tau^\ell\|_{\dot{B}_{2,1}^{\frac n2-1}}+\|\tau^h\|_{\dot{B}_{p,1}^{\frac np}})(\| u^\ell\|_{\dot B^{\frac  n2+1}_{2,1}}  +\|u^h\|_{\dot B^{\frac  np+1}_{p,1}}).
\end{align}
The term $\| (F(\tau, \nabla u))^\ell\|_{\dot{B}_{2,1}^{\frac n2-1}}$ can be dealt with the same method as \eqref{danian7}, as a result, we can get
\begin{align}\label{G26+9912}
 \|f^\ell\|_{\dot{B}_{2,1}^{\frac n2-1}}
\lesssim&(\|\tau^\ell\|_{\dot{B}_{2,1}^{\frac n2-1}}+\|\tau^h\|_{\dot{B}_{p,1}^{\frac np}})(\| u^\ell\|_{\dot B^{\frac  n2+1}_{2,1}}  +\|u^h\|_{\dot B^{\frac  np+1}_{p,1}}).
\end{align}

\noindent Thanks to Corollary \ref{xinqing}, we obtain
\begin{align}\label{G262324}
 \|g^\ell\|_{\dot{B}_{2,1}^{\frac n2-1}}\lesssim&\|([\p ,u\cdot \nabla] u)^\ell\|_{\dot{B}_{2,1}^{\frac n2-1}}
 \lesssim
 \| u\|_{\dot{B}_{p,1}^{\frac np-1}}\|u\|_{\dot{B}_{p,1}^{\frac np+1}}\nonumber\\
  \lesssim&(\|u^\ell\|_{\dot B^{-1+\frac  n2}_{2,1}}+\|u^h\|_{\dot B^{\frac  np-1}_{p,1}})(\| u^\ell\|_{\dot B^{\frac  n2+1}_{2,1}}  +\|u^h\|_{\dot B^{\frac  np+1}_{p,1}}).
 \end{align}
We get by a similar derivation of \eqref{D9}--\eqref{D11} that
\begin{align}\label{G27}
\|([\La, u\cdot \nabla]\aa)^\ell\|_{\dot{B}_{2,1}^{\frac n2-1}}
\lesssim&
\|(\La u\cdot \nabla\aa)^\ell\|_{\dot{B}_{2,1}^{\frac n2-1}}
\nonumber\\
\lesssim& (\| \aa^\ell\|_{\dot B^{\frac  n2-1}_{2,1}}  +\|\aa^h\|_{\dot B^{\frac  np}_{p,1}})(\|u^\ell\|_{\dot B^{\frac  n2+1}_{2,1}}+\|u^h\|_{\dot B^{\frac  np+1}_{p,1}}).
\end{align}
By using Lemma \ref{jiaohuanzi}, we can get
\begin{align}\label{G281}
&\int^t_0\sum_{j\le j_0}2^{(\frac n2-1)j}(\|[\ddj,u\cdot \nabla]\aa\|_{L^2}+\|[\ddj,u\cdot \nabla]u\|_{L^2}+\|[\ddj,u\cdot \nabla]\La\aa\|_{L^2})\,ds\nonumber\\
&\quad\lesssim\int^t_0(\| \aa^\ell\|_{\dot B^{\frac  n2-1}_{2,1}}  +\|\aa^h\|_{\dot B^{\frac  np}_{p,1}})(\|u^\ell\|_{\dot B^{\frac  n2+1}_{2,1}}+\|u^h\|_{\dot B^{\frac  np+1}_{p,1}})\,ds\nonumber\\
&\quad\quad+\int^t_0
(\|u^\ell\|_{\dot B^{-1+\frac  n2}_{2,1}}+\|u^h\|_{\dot B^{\frac  np-1}_{p,1}})(\| u^\ell\|_{\dot B^{\frac  n2+1}_{2,1}}  +\|u^h\|_{\dot B^{\frac  np+1}_{p,1}})\,ds.
\end{align}

Inserting \eqref{G26+9912}, \eqref{G262324}, \eqref{G27} and  \eqref{G281} into \eqref{G1125} gives
\begin{align}\label{G15678}
&\|(u^\ell,\aa^\ell)\|_{\widetilde{L}_t^{\infty}(\dot{B}_{2,1}^{\frac n2-1})}
+\int^t_0\| (u^\ell,\aa^\ell)\|_{\dot{B}_{2,1}^{\frac n2+1}}\,ds\nonumber\\
&\quad\lesssim\|(u_0^\ell,\aa_0^\ell)\|_{\dot{B}_{2,1}^{\frac {n}{2}-1}}
\nonumber\\
&\quad\quad+\int^t_0(\|(u^\ell,\tau^\ell)\|_{\dot B^{\frac  n2-1}_{2,1}}+\|u^h\|_{\dot B^{\frac  np-1}_{p,1}}+\|\tau^h\|_{\dot{B}_{p,1}^{\frac np}})(\| u^\ell\|_{\dot B^{\frac  n2+1}_{2,1}}  +\|u^h\|_{\dot B^{\frac  np+1}_{p,1}})\,ds.
\end{align}

Combining with \eqref{nian8} and \eqref{G15678}, we can get
\begin{align}\label{G1598}
&\|(u^\ell,\tau^\ell)\|_{\widetilde{L}_t^{\infty}(\dot{B}_{2,1}^{\frac n2-1})}+\|\tau^\ell\|_{\widetilde{L}_t^{\infty}(\dot{B}_{2,1}^{\frac n2})}+\int^t_0\| u^\ell\|_{\dot{B}_{2,1}^{\frac n2+1}}\,ds
+\int^t_0\| (\Lambda^{-1}\p\div\tau)^\ell\|_{\dot{B}_{2,1}^{\frac n2+1}}\,ds\nonumber\\
&\quad\lesssim\|(u_0^\ell,\tau_0^\ell)\|_{\dot{B}_{2,1}^{\frac {n}{2}-1}}
\nonumber\\
&\quad\quad+\int^t_0(\|(u^\ell,\tau^\ell)\|_{\dot B^{\frac  n2-1}_{2,1}}+\|u^h\|_{\dot B^{\frac  np-1}_{p,1}}+\|\tau^h\|_{\dot{B}_{p,1}^{\frac np}})(\| u^\ell\|_{\dot B^{\frac  n2+1}_{2,1}}  +\|u^h\|_{\dot B^{\frac  np+1}_{p,1}})\,ds.
\end{align}

\subsection{The high frequency estimates of the solutions }

In the following, we are concerned with the estimates for the high frequency part of the solution.
We shall find the damping effect of $\aa$ and smoothing effect of $u$ in the high frequency part.

Let $\ga=u-\La^{-1}\aa$, we can get by a simple computation from \eqref{G5} that
\begin{eqnarray}\label{asG5}
\left\{\begin{aligned}
&\partial_t \aa+ u\cdot \nabla\aa+\aa=f -\La\ga,\\
&\partial_t  \ga -\Delta  \ga=\ga+\La^{-1}\aa-\p (u\cdot \nabla u)-\Lambda^{-1}(u\cdot \nabla\aa)-\Lambda^{-1}f.
\end{aligned}\right.
\end{eqnarray}
Applying $\ddj $ to the first equation in \eqref{asG5} and taking $L^2$ inner product with $|\ddj\aa|^{p-2}\ddj\aa$, using integrating by part and  the H\"older inequality,  we thus get for all $t\geq0,$
\begin{align}\label{W1}
&\|\ddj \aa(t)\|_{L^p}+\int_0^t\|\ddj \aa\|_{L^p}\,ds\nonumber\\
&\quad\leq\|\ddj \aa_0\|_{L^p}
+\int_0^t(\|[\ddj,u\cdot\nabla]\aa\|_{L^p}+\|\ddj f\|_{L^p}+\|\ddj(\La\ga) \|_{L^p})\,ds
\end{align}
from which and  the definition of Besovs spaces that
\begin{align}\label{asA2}
\|\aa^h\|_{ \widetilde{L}_t^\infty(\dot B^{\frac  np}_{p,1})}+ \|\aa^h\|_{L^1_t(\dot B^{\frac  np}_{p,1})}
\lesssim& \|\aa_0^h\|_{\dot B^{\frac  np}_{p,1}}
+\sum_{j\ge j_0}2^{\frac {nj}{p}}\|[\ddj,u\cdot\nabla]\aa\|_{L^1_t(L^p)}\nonumber\\
&+ \|f^h\|_{L^1_t(\dot B^{\frac  np}_{p,1})}+ \|\ga^h\|_{L^1_t(\dot B^{\frac  np+1}_{p,1})}.
\end{align}
Similarly, we  get the high frequency of $\ga$ that
 \begin{align}\label{asaA2}
&\|\ga^h\|_{ \widetilde{L}_t^\infty(\dot B^{\frac  np-1}_{p,1})}+ \|\ga^h\|_{L^1_t(\dot B^{\frac  np+1}_{p,1})}\nonumber\\
&\quad\lesssim \|\ga_0^h\|_{\dot B^{\frac  np-1}_{p,1}}
+ \|\ga^h\|_{L^1_t(\dot B^{\frac  np-1}_{p,1})}+ \|(\La^{-1}\aa)^h\|_{L^1_t(\dot B^{\frac  np-1}_{p,1})}\nonumber\\
&\quad\quad+ \|(\p (u\cdot \nabla u))^h\|_{L^1_t(\dot B^{\frac  np-1}_{p,1})}+ \|(\Lambda^{-1}f)^h\|_{L^1_t(\dot B^{\frac  np-1}_{p,1})}+ \|(\Lambda^{-1}(u\cdot \nabla\aa))^h\|_{L^1_t(\dot B^{\frac  np-1}_{p,1})}\nonumber\\
&\quad\lesssim \|\ga_0^h\|_{\dot B^{\frac  np-1}_{p,1}}
+ 2^{-2j_0}\|\ga^h\|_{L^1_t(\dot B^{\frac  np+1}_{p,1})}+2^{-2j_0}\|\aa^h\|_{L^1_t(\dot B^{\frac  np}_{p,1})}\nonumber\\
&\quad\quad+ \|(\p (u\cdot \nabla u))^h\|_{L^1_t(\dot B^{\frac  np-1}_{p,1})}+ \|f^h\|_{L^1_t(\dot B^{\frac  np-1}_{p,1})}+ \|(\Lambda^{-1}(u\cdot \nabla\aa))^h\|_{L^1_t(\dot B^{\frac  np-1}_{p,1})}.
\end{align}

Combining with \eqref{asA2} and \eqref{asaA2}, one can deduce from $u=\ga+\La^{-1}\aa$ that
\begin{align}\label{laA2}
&\|u^h\|_{ \widetilde{L}_t^\infty(\dot B^{\frac  np-1}_{p,1})}+\|\aa^h\|_{ \widetilde{L}_t^\infty(\dot B^{\frac  np}_{p,1})}+ \|\aa^h\|_{L^1_t(\dot B^{\frac  np}_{p,1})}+ \|u^h\|_{L^1_t(\dot B^{\frac  np+1}_{p,1})}\nonumber\\
&\quad\lesssim \|u_0^h\|_{\dot B^{\frac  np-1}_{p,1}}+\|\aa_0^h\|_{\dot B^{\frac  np}_{p,1}}+ \|f^h\|_{L^1_t(\dot B^{\frac  np}_{p,1})}+\sum_{j\ge j_0}2^{\frac {nj}{p}}\|[\ddj,u\cdot\nabla]\aa\|_{L^1_t(L^p)}\nonumber\\
&\quad\quad+ \|(\p (u\cdot \nabla u))^h\|_{L^1_t(\dot B^{\frac  np-1}_{p,1})}+ \|(\Lambda^{-1}(u\cdot \nabla\aa))^h\|_{L^1_t(\dot B^{\frac  np-1}_{p,1})}.
\end{align}
With the aid of Lemmas  \ref{daishu} and \ref{jiaohuanzi}, we have
\begin{align}\label{la1}
\|f^h\|_{L^1_t(\dot B^{\frac  np}_{p,1})}
\lesssim&\int_0^t\|[\Lambda^{-1}\p\div,u\cdot \nabla ]\tau\|_{\dot B^{\frac  np}_{p,1}}+\|\Lambda^{-1}\p \div(F(\tau, \nabla u))\|_{\dot B^{\frac  np}_{p,1}} \,ds\nonumber\\
\lesssim&\int_0^t\|\tau\|_{\dot B^{\frac  np}_{p,1}}\|\nabla u\|_{\dot B^{\frac  np}_{p,1}}  \,ds\nonumber\\
\lesssim&\int_0^t(\|\tau^\ell\|_{\dot B^{-1+\frac  n2}_{2,1}}+\|\tau^h\|_{\dot B^{\frac  np}_{p,1}})(\| u^\ell\|_{\dot B^{\frac  n2+1}_{2,1}}  +\|u^h\|_{\dot B^{\frac  np+1}_{p,1}}) \,ds.
\end{align}
Similarly,
\begin{align}\label{la3-1}
\sum_{j\ge j_0}2^{\frac {nj}{p}}\|[\ddj,u\cdot\nabla]\aa\|_{L^1_t(L^p)}
\lesssim&\int_0^t\|\nabla u\|_{\dot B^{\frac  np}_{p,1}}\|\aa\|_{\dot B^{\frac  np}_{p,1}}  \,ds\nonumber\\
\lesssim&\int_0^t(\|\tau^\ell\|_{\dot B^{-1+\frac  n2}_{2,1}}+\|\tau^h\|_{\dot B^{\frac  np}_{p,1}})(\| u^\ell\|_{\dot B^{\frac  n2+1}_{2,1}}  +\|u^h\|_{\dot B^{\frac  np+1}_{p,1}}) \,ds,\nonumber\\
\|(\p (u\cdot \nabla u))^h\|_{L^1_t(\dot B^{\frac  np-1}_{p,1})}
\lesssim&\int_0^t\|u\|_{\dot B^{\frac  np-1}_{p,1}}\|\nabla u\|_{\dot B^{\frac  np}_{p,1}}  ds\nonumber\\
\lesssim&\int_0^t(\|u^\ell\|_{\dot B^{-1+\frac  n2}_{2,1}}+\|u^h\|_{\dot B^{\frac  np-1}_{p,1}})(\| u^\ell\|_{\dot B^{\frac  n2+1}_{2,1}}  +\|u^h\|_{\dot B^{\frac  np+1}_{p,1}}) \,ds.
\end{align}
By Lemma \ref{daishu}, one has
\begin{align}\label{la3}
\|(\Lambda^{-1}(u\cdot \nabla\aa))^h\|_{L^1_t(\dot B^{\frac  np-1}_{p,1})}
\lesssim&\int_0^t\|\Lambda^{-1} \div(u\aa)\|_{\dot B^{\frac  np}_{p,1}}\,ds
\lesssim\int_0^t\| u\|_{\dot B^{\frac  np}_{p,1}}\|\aa\|_{\dot B^{\frac  np}_{p,1}}\,ds\nonumber\\
\lesssim&\int_0^t(\| u\|_{\dot B^{\frac  np}_{p,1}}^2+\|\aa^\ell\|_{\dot B^{\frac  n2}_{2,1}}^2+\|\aa^h\|_{\dot B^{\frac  np}_{p,1}}^2)\,ds\nonumber\\
\lesssim&\int_0^t(\| u\|_{\dot B^{\frac  np-1}_{p,1}}\| u\|_{\dot B^{\frac  np+1}_{p,1}}+\|\aa^\ell\|_{\dot B^{\frac  n2-1}_{2,1}}\|\aa^\ell\|_{\dot B^{\frac  n2+1}_{2,1}}+\|\aa^h\|_{\dot B^{\frac  np}_{p,1}}^2)\,ds\nonumber\\
\lesssim&\int_0^t(\|u^\ell\|_{\dot B^{-1+\frac  n2}_{2,1}}+\|u^h\|_{\dot B^{\frac  np-1}_{p,1}})(\| u^\ell\|_{\dot B^{\frac  n2+1}_{2,1}}  +\|u^h\|_{\dot B^{\frac  np+1}_{p,1}}) \,ds\nonumber\\
&+\int_0^t(\|\tau^\ell\|_{\dot B^{-1+\frac  n2}_{2,1}}+\|\tau^h\|_{\dot B^{\frac  np}_{p,1}})(\| \aa^\ell\|_{\dot B^{\frac  n2+1}_{2,1}}  +\|\aa^h\|_{\dot B^{\frac  np}_{p,1}}) \,ds.
\end{align}

Plugging \eqref{la1}--\eqref{la3} into \eqref{laA2} implies
\begin{align}\label{maaA2}
&\|u^h\|_{ \widetilde{L}_t^\infty(\dot B^{\frac  np-1}_{p,1})}+\|\aa^h\|_{ \widetilde{L}_t^\infty(\dot B^{\frac  np}_{p,1})}+ \|\aa^h\|_{L^1_t(\dot B^{\frac  np}_{p,1})}+ \|u^h\|_{L^1_t(\dot B^{\frac  np+1}_{p,1})}\nonumber\\
&\quad\lesssim \|u_0^h\|_{\dot B^{\frac  np-1}_{p,1}}+\|\aa_0^h\|_{\dot B^{\frac  np}_{p,1}}\nonumber\\
&\quad\quad
+\int_0^t(\|\tau^\ell\|_{\dot B^{-1+\frac  n2}_{2,1}}+\|\tau^h\|_{\dot B^{\frac  np}_{p,1}})(\| \aa^\ell\|_{\dot B^{\frac  n2+1}_{2,1}}  +\|\aa^h\|_{\dot B^{\frac  np}_{p,1}}) \,ds\nonumber\\
&\quad\quad+\int_0^t(\|(u^\ell,\tau^\ell)\|_{\dot B^{\frac  n2-1}_{2,1}}+\|u^h\|_{\dot B^{\frac  np-1}_{p,1}}+\|\tau^h\|_{\dot{B}_{p,1}^{\frac np}})(\| u^\ell\|_{\dot B^{\frac  n2+1}_{2,1}}  +\|u^h\|_{\dot B^{\frac  np+1}_{p,1}})\,ds.
\end{align}
 From the first equation in \eqref{m}, we can get similarly to \eqref{asA2} that
\begin{align}\label{xihuan1}
\|\tau^h\|_{\widetilde{L}_t^{\infty}(\dot{B}_{p,1}^{\frac np})}
\lesssim&\|\tau_0^h\|_{\dot{B}_{p,1}^{\frac {n}{p}}}+\int^t_0\| u^h\|_{\dot{B}_{p,1}^{\frac np+1}}\,ds+
\int^t_0\|\nabla u\|_{\dot{B}_{p,1}^{\frac {n}{p}}}\|\tau\|_{\dot{B}_{p,1}^{\frac np}}\,ds.
\end{align}

Together with \eqref{maaA2} and \eqref{xihuan1}, one has
\begin{align}\label{xihuan2}
&\|u^h\|_{ \widetilde{L}_t^\infty(\dot B^{\frac  np-1}_{p,1})}+\|\tau^h\|_{ \widetilde{L}_t^\infty(\dot B^{\frac  np}_{p,1})}+ \|(\Lambda^{-1}\p\div\tau)^h\|_{L^1_t(\dot B^{\frac  np}_{p,1})}+ \|u^h\|_{L^1_t(\dot B^{\frac  np+1}_{p,1})}\nonumber\\
&\quad\lesssim \|u_0^h\|_{\dot B^{\frac  np-1}_{p,1}}+\|\tau_0^h\|_{\dot B^{\frac  np}_{p,1}}
\nonumber\\
&\quad\quad
+\int_0^t(\|\tau^\ell\|_{\dot B^{-1+\frac  n2}_{2,1}}+\|\tau^h\|_{\dot B^{\frac  np}_{p,1}})(\| \aa^\ell\|_{\dot B^{\frac  n2+1}_{2,1}}  +\|\aa^h\|_{\dot B^{\frac  np}_{p,1}}) \,ds\nonumber\\
&\quad\quad+\int_0^t(\|(u^\ell,\tau^\ell)\|_{\dot B^{\frac  n2-1}_{2,1}}+\|u^h\|_{\dot B^{\frac  np-1}_{p,1}}+\|\tau^h\|_{\dot{B}_{p,1}^{\frac np}})(\| u^\ell\|_{\dot B^{\frac  n2+1}_{2,1}}  +\|u^h\|_{\dot B^{\frac  np+1}_{p,1}})\,ds.
\end{align}

\subsection{ Complete the proof of our main Theorem \ref{zhuyaodingli}}
Now, we can complete the proof of our main Theorem \ref{zhuyaodingli} by the continuous arguments.
Denote
\begin{align*}
X(t)\stackrel{\mathrm{def}}{=}&\|(u^\ell,\tau^\ell)\|_{\widetilde{L}_t^{\infty}(\dot{B}_{2,1}^{\frac n2-1})}+\| u^\ell\|_{L^1_t(\dot{B}_{2,1}^{\frac n2+1})}
+\| (\Lambda^{-1}\p\div\tau)^\ell\|_{L^1_t(\dot{B}_{2,1}^{\frac n2+1})}\nonumber\\
&\quad+\|u^h\|_{ \widetilde{L}_t^\infty(\dot B^{\frac  np-1}_{p,1})}+\|\tau^h\|_{ \widetilde{L}_t^\infty(\dot B^{\frac  np}_{p,1})}+ \|(\Lambda^{-1}\p\div\tau)^h\|_{L^1_t(\dot B^{\frac  np}_{p,1})}+ \|u^h\|_{L^1_t(\dot B^{\frac  np+1}_{p,1})}.
\end{align*}

Combining with \eqref{G1598} and \eqref{xihuan2}, we can get
\begin{align}\label{G31}
&\|(u^\ell,\tau^\ell)\|_{\widetilde{L}_t^{\infty}(\dot{B}_{2,1}^{\frac n2-1})}+\| u^\ell\|_{L^1_t(\dot{B}_{2,1}^{\frac n2+1})}
+\| (\Lambda^{-1}\p\div\tau)^\ell\|_{L^1_t(\dot{B}_{2,1}^{\frac n2+1})}\nonumber\\
&\quad\quad+\|u^h\|_{ \widetilde{L}_t^\infty(\dot B^{\frac  np-1}_{p,1})}+\|\tau^h\|_{ \widetilde{L}_t^\infty(\dot B^{\frac  np}_{p,1})}+ \|(\Lambda^{-1}\p\div\tau)^h\|_{L^1_t(\dot B^{\frac  np}_{p,1})}+ \|u^h\|_{L^1_t(\dot B^{\frac  np+1}_{p,1})}\nonumber\\
&\quad\lesssim\|(u_0^\ell,\tau_0^\ell)\|_{\dot{B}_{2,1}^{\frac {n}{2}-1}}+ \|u_0^h\|_{\dot B^{\frac  np-1}_{p,1}}+\|\tau_0^h\|_{\dot B^{\frac  np}_{p,1}}
\nonumber\\
&\quad\quad
+\int_0^t(\|\tau^\ell\|_{\dot B^{-1+\frac  n2}_{2,1}}+\|\tau^h\|_{\dot B^{\frac  np}_{p,1}})(\| (\Lambda^{-1}\p\div\tau)^\ell\|_{\dot B^{\frac  n2+1}_{2,1}}  +\|(\Lambda^{-1}\p\div\tau)^h\|_{\dot B^{\frac  np}_{p,1}}) \,ds\nonumber\\
&\quad\quad+\int_0^t(\|(u^\ell,\tau^\ell)\|_{\dot B^{\frac  n2-1}_{2,1}}+\|u^h\|_{\dot B^{\frac  np-1}_{p,1}}+\|\tau^h\|_{\dot{B}_{p,1}^{\frac np}})(\| u^\ell\|_{\dot B^{\frac  n2+1}_{2,1}}  +\|u^h\|_{\dot B^{\frac  np+1}_{p,1}})\,ds.
\end{align}
 From \eqref{G31} and the Gronwall inequality, we have
\begin{align}\label{yi6}
X(t)\le Ce^{CX(t)}(\|(u_0^\ell,\tau_0^\ell)\|_{\dot{B}_{2,1}^{\frac {n}{2}-1}}+ \|u_0^h\|_{\dot B^{\frac  np-1}_{p,1}}+\|\tau_0^h\|_{\dot B^{\frac  np}_{p,1}}).
\end{align}
 Now let $\delta$ be a
 positive constant, which will
 be determined later on. For any $T^\flat\in[0,T^\ast),$ we define
 \begin{align*}
T^{\ast\ast}\stackrel{\mathrm{def}}{=}\sup\Big\{
t\in [0,T^\flat):
X(t)
\leq \delta\Big\}.
\end{align*}
From \eqref{yi6}, we have for any  $t\in[0,T^{\ast\ast})$ there holds
\begin{align}\label{yi69}
X(t)\le C_1e^{{C_1\delta}}(\|(u_0^\ell,\tau_0^\ell)\|_{\dot{B}_{2,1}^{\frac {n}{2}-1}}+ \|u_0^h\|_{\dot B^{\frac  np-1}_{p,1}}+\|\tau_0^h\|_{\dot B^{\frac  np}_{p,1}}).
\end{align}
Choosing $\delta<\frac{1}{4C_1}$ fixed and then letting $$\|(u_0^\ell,\tau_0^\ell)\|_{\dot{B}_{2,1}^{\frac {n}{2}-1}}+ \|u_0^h\|_{\dot B^{\frac  np-1}_{p,1}}+\|\tau_0^h\|_{\dot B^{\frac  np}_{p,1}}<\frac{1}{8C_1},$$ we can get from
\eqref{yi69} that
\begin{align*}
X(t)\le \frac{\delta}{2}, \quad \forall t\in [0,T^{\ast\ast}],
\end{align*}
 this contradicts with the definition of $T^{\ast\ast}$, thus we conclude that $T^{\ast\ast}=T^{\ast}$.
Consequently, we complete the proof of Theorem \ref{zhuyaodingli} by standard continuation argument. \hspace{2cm} $\square$

\bigskip
\noindent{\large\bf Conflict of interest}
This work does not have any conflicts of interest.

 \bigskip
\noindent {\large\bf Acknowledgement} {This work is partially
 supported by the NSFC (11601533).}

\newpage

\end{document}